\documentclass[11pt]{amsart}
\usepackage{amsfonts, amssymb, amsmath}
\usepackage{verbatim}
\usepackage{mathrsfs}
\usepackage{enumerate,accents}
\usepackage{tikz-cd}
\usepackage{url}
\usepackage{enumitem}
\usepackage[hidelinks]{hyperref}
\usepackage[all]{xy}
\usepackage[utf8]{inputenc}

\newcommand{\PP}{\mathbb{P}}
\newcommand{\F}{\mathbb{F}}

\newcommand{\KK}{\mathbb{K}}

\newcommand{\Q}{\mathbb{Q}}
\newcommand{\Z}{\mathbb{Z}}

\newcommand{\iso}{\xrightarrow{\sim}}
\newcommand{\hZ}{\widehat{\Z}}

\newcommand{\Fq}{\F_q}

\newcommand{\Qlbar}{\overline{\mathbb{Q}}_{\ell}}

\newcommand{\et}{\mathrm{\acute{e}t}}
\newcommand{\pie}{\pi_1^{\et}}

\newcommand{\calP}{\mathcal{P}}

\newcommand{\calV}{\mathcal{V}}

\newcommand{\ucalVz}{\underline{\mathcal{V}_0}}
\newcommand{\calW}{\mathcal{W}}

\DeclareMathOperator{\Ker}{Ker}

\DeclareMathOperator{\Gal}{Gal}

\newcommand{\Image}{\mathrm{Im}}

\def\et{\mathrm{\acute{e}t}}

\begin{document}
\bibliographystyle{alpha}
\newtheorem{theorem}[subsubsection]{Theorem}
\newtheorem*{theorem*}{Theorem}
\newtheorem{conjecture}[subsubsection]{Conjecture}
\newtheorem{proposition}[subsubsection]{Proposition}
\newtheorem{lemma}[subsubsection]{Lemma}
\newtheorem*{lemma*}{Lemma}
\newtheorem{corollary}[subsubsection]{Corollary}

\theoremstyle{definition}
\newtheorem{definition}[subsubsection]{Definition}
\newtheorem*{definition*}{Definition}
\newtheorem{property}[subsubsection]{Property}

\newtheorem{rema}[subsubsection]{Remark}
\newtheorem{example}[subsubsection]{Example}
\newtheorem{notation}[subsubsection]{Notation}
\newtheorem{construction}[subsubsection]{Construction}

\setcounter{tocdepth}{2}

\title{Some remarks on the companions conjecture for normal varieties}

\author{Marco D'Addezio}
\date{\today}

\address{Max-Planck-Institut für Mathematik, Vivatsgasse 7, 53111, Bonn, Germany}
\email{daddezio@mpim-bonn.mpg.de}

\begin{abstract}
Drinfeld in 2010 proved the companions conjecture for smooth varieties over a finite field, generalizing L. Lafforgue's result for smooth curves. We study the obstruction to prove the conjecture for arbitrary normal varieties. To do this, we introduce a new property of morphisms. We verify this property in some cases, showing thereby the companions conjecture for some singular normal varieties. 
\end{abstract}
\maketitle
\tableofcontents

\section{Introduction}
\subsection{The companions conjecture}
Let $\Fq$ be a finite field of characteristic $p$ and $X_0$ a connected normal variety over $\Fq$. Let $\ell$ be a prime different from $p$ and $\calV_0$ an irreducible Weil lisse $\Qlbar$-sheaf over $X_0$ with finite order determinant. Write $E$ for the subfield of $\Qlbar$ generated by the coefficients of all these Frobenius polynomials at closed points. Deligne proved that $E$ is a finite extension of $\Q$, \cite{Del}. He showed this finiteness reducing geometrically to the case of curves, which was already proven by L. Lafforgue as a consequence of the Langlands correspondence, \cite{Laf}. This property of the field $E$ was conjectured by Deligne in \cite[Conjecture 1.2.10]{Weil2} together with other properties for $\calV_0$. The principal part of the conjecture is the following one.

\begin{conjecture}[Companions conjecture]\label{intro-companions-c}After possibly replacing $E$ with a finite extension, for every finite place $\lambda$ not dividing $p$ there exists a Weil lisse $E_{\lambda}$-sheaf $E$-compatible\footnote{Cf. \cite[Definition 3.1.15]{Dad}.} with $\calV_0$.
\end{conjecture}

When $X_0$ has dimension $1$, the conjecture is again a consequence of the Langlands correspondence. For higher dimensional varieties, Drinfeld proved Conjecture \ref{intro-companions-c} when $X_0$ is smooth, \cite{Dri}. Unluckily, his method cannot be applied directly to prove the full conjecture, [\textit{ibid.}, §6].

\subsection{The obstruction}\label{obst:ss}
Suppose for simplicity that the singular locus of $X_0$ consists of one closed point and that we can solve that singularity. In other words, suppose that there exists a smooth variety $Y_0$ and a proper morphism $h_0:Y_0\to X_0$ sending $i_0:Z_0\hookrightarrow Y_0$ to a closed point $x_0\in |X_0|$ and such that $h_0$ is an isomorphism outside $Z_0$. Write $\F$ for an algebraic closure of $\Fq$ and suppose that $Z:=Z_0\otimes_{\Fq}\F$ is connected.
\begin{lemma}[Corollaire 6.11, SGA I, Exposé XI]

\label{i-contraction-fundamental-group-lemma} For every geometric point $z$ of $Z_0$ there exists an exact sequence $$\pi_1^{\et}(Z,z)\xrightarrow{i_*} \pi_1^{\et}(Y_0,z)\xrightarrow{h_{0*}} \pi_1^{\et}(X_0,h(z))\to 1$$ in the sense that the smallest normal closed subgroup containing the image of $i_*$ is the kernel of $h_{0*}$.
\end{lemma}
By the lemma, every étale Weil lisse sheaf $\calV_0$ over $Y_0$ which is trivial over $Z$ is the inverse image of an étale Weil lisse sheaf defined over $X_0$. 
Since we know the companions conjecture for $Y_0$, in order to deduce it for $X_0$ we have to verify the following property. 

\begin{itemize}
	\item[$\mathcal{P}(Z_0)$ :] For every pair $(\calV_0,\calW_0)$ of compatible absolutely irreducible Weil lisse sheaves with finite order determinant over $Y_0$, the sheaf $\calV_0$ is trivial over $Z$ if and only if the same is true for $\calW_0$.
\end{itemize}
If $Z_0\subseteq Y_0$ satisfies $\calP(Z_0)$ we say that $Z_0$ is a \textit{$\lambda$-uniform subvariety}. Thanks to Lemma \ref{i-contraction-fundamental-group-lemma} and the companions conjecture for smooth varieties, we have the following result.

\begin{theorem}\label{lamd-unif-comp:t}
If $Z_0\subseteq Y_0$ is $\lambda$-uniform, $X_0$ satisfies the companions conjecture.
\end{theorem}
\subsection{Main results}

The aim of this text is to shed some new lights on the companions conjecture for normal varieties. For this scope we focus on $\lambda$-uniformity. We extend the notion of $\lambda$-uniform subvarieties of §\ref{obst:ss} with the one of \textit{$\lambda$-uniform morphisms} of varieties (Definition \ref{uniform-morphisms:d}). We intend to investigate the following conjecture. 
\begin{conjecture}[Conjecture \ref{unif:c}]\label{i-unif:c}
Let $Y_0$ and $Z_0$ be varieties over $\Fq$. If $Y_0$ is normal, every morphism $f_0:Z_0\to Y_0$ is $\lambda$-uniform (cf. Definition \ref{uniform-morphisms:d}).
\end{conjecture}

We shall verify Conjecture \ref{i-unif:c} in some particular cases.

\begin{theorem}[Theorem \ref{normal-uniform:t}, Theorem \ref{tree:t}, and 
	Theorem \ref{finite:t}]\label{i-main:t}
Let $f_0: Z_0 \to Y_0$ be a morphism of geometrically connected varieties over $\Fq$ with $Y_0$ normal. Let $z$ be a geometric point of $Z_0$. The morphism $f_0$ is $\lambda$-uniform in the following cases.
\begin{itemize}
\item[{\rm (i)}] If $Z_0$ is a normal variety.
\item[{\rm (ii)}] If $Z_0$ is a semi-stable curve with simply connected dual graph.
\item[{\rm (iii)}]  If the smallest closed normal subgroup of $\pi_1^{\et}(Y_0,f(z))$ containing the image of $\pi_1^{\et}(Z,z)$ is open inside $\pi_1^{\et}(Y,f(z))$.
\item[{\rm (iv)}] If $\pi_1^{\et}(Y,f(z))$ contains an open solvable profinite subgroup.
\end{itemize}
\end{theorem}

Combining the previous results we get the following.
\begin{theorem}Let $Y_0$ be a smooth geometrically connected variety over $\Fq$. If $X_0$ be a normal variety that can be written as a contraction of a geometrically connected subvariety $Z_0\subseteq Y_0$ satisfying one of the conditions of Theorem \ref{i-main:t}, then $X_0$ verifies Conjecture \ref{intro-companions-c}.   
\end{theorem}
%Combining Proposition \ref{lamd-unif-comp:p}, Theorem \ref{a-main:t} and \cite{Dri}, we deduce the companions conjecture for those normal varieties obtained by contracting a subvariety $Z_0\subseteq Y_0$ satisfying one of the four previous conditions. 

An independent property we prove in this text is a property of invariance of $\lambda$-uniformity “under deformations” (Theorem \ref{homotopy:t}). This might be useful for further developments in the direction of Conjecture \ref{i-unif:c}. Besides, we present in §\ref{fina-comm:ss} a concrete example, proposed by de Jong, where we think it might be interesting to study Conjecture \ref{i-unif:c}.

\subsection{Acknowledgements}
I am grateful to my advisor Hélène Esnault for introducing me to this topic and for all the time we spent talking about this problem. I thank Emiliano Ambrosi, Raju Krishnamoorthy, and Jacob Stix for discussions and suggestions. Finally, I thank Piotr Achinger and Daniel Litt for sharing with me de Jong's Example \ref{dJ:ex}.

\subsection{Notation and conventions}
\subsubsection{}For us, a variety over a field $k$ is a separated scheme of finite type over $k$. We write $X_0,Y_0,Z_0,...$ for varieties over $\F_q$ and $X,Y,Z,...$ for the base change to $\F$. Further, we put a subscript $_0$ to indicate objects and morphisms defined over $\F_q$ and the suppression of this subscript shall mean that we are extending the scalars to $\F$. If $E$ is a number field we write $|E|_{\neq p}$ for the set of finite places of $E$ which do not divide $p$. For every $\lambda\in |E|_{\neq p}$, we denote by $E_{\lambda}$ the completion of $E$ with respect to $\lambda$.

\subsubsection{}
We use the notation for Weil lisse sheaves as in \cite[§2.2]{Dad}. We say that a Weil lisse $E_\lambda$-sheaf $\calV_0$ is \textit{split untwisted} if every subquotient of $\calV_0$ is absolutely irreducible and has finite order determinant. We say instead that $\calV_0$ is \textit{untwisted} if it is split untwisted after possibly extending $E_\lambda$. Recall that if $Y_0$ is a normal variety over $\Fq$, every untwisted Weil lisse sheaf over $Y_0$ is pure of weight $0$, \cite[Théorème 1.6]{Del}. Besides, by \cite[Proposition 1.3.14]{Weil2} and \cite[Proposition 3.1.16]{Dad}, every untwisted Weil lisse sheaf is étale.

\subsubsection{}

An \textit{$E$-compatible system} over $X_0$, denoted by $\underline{\calV_0}$, is a family $\{\calV_{\lambda,0}\}_{\lambda\in |E|_{\neq p}}$ where each $\calV_{\lambda,0}$ is an $E$-rational Weil lisse $E_{\lambda}$-sheaf and such that all sheaves are pairwise $E$-compatible. Each $\calV_{\lambda,0}$ is called the \textit{$\lambda$-component} of $\underline{\calV_0}$. We say that a compatible system is \textit{semi-simple, untwisted, split untwisted}, ... if each $\lambda$-component has the respective property.

\section{\texorpdfstring{$\lambda$}{}-uniform morphisms}

%Let $X_0$ be a normal scheme of finite type over $\F_q$. We want to study the following Deligne's conjecture.

%\begin{conjecture}[Companions conjecture]\label{companion-normal-conjecture} Let $E$ be a number field. For every pair of finite places $\lambda$ and $\lambda'$ not dividing $p$ and for every $E$-rational lisse $\overline{E}_{\lambda}$-sheaf there exists an $E$-compatible lisse $\overline{E}_{\lambda'}$-sheaf.
%\end{conjecture}

%\begin{theorem}[Drinfeld]
%The companions conjecture holds on smooth varieties.
%\end{theorem}
%\begin{proof}
%\cite{Dri}
%\end{proof}
%Unluckily Drinfeld's proof does not apply for general normal varieties, as it is also explained in [\textit{ibid.}~, §6].
%We will focus on the reduction of this conjecture to smooth varieties. We present now the essential obstruction to reduce to the case of smooth varieties.

\label{the-conj:ss}
\subsection{General properties}
\begin{definition}
\label{uniform-systems:d}
Let $Z_0$ be a connected variety. A compatible system $\underline{\calV_0}$ over $Z_0$ is \textit{$\lambda$-uniform} if one of the following disjoint conditions is verified.

\begin{itemize}
\item [(i)] For every $\lambda\nmid p$, the lisse sheaf $\calV_{\lambda,0}$ is geometrically trivial.
\item [(ii)]For every $\lambda\nmid p$, the lisse sheaf $\calV_{\lambda,0}$ is geometrically non-trivial.
\end{itemize}
We say that $\underline{\calV_0}$ is \textit{strongly $\lambda$-uniform} if the dimension of $H^0(Z,\calV_{\lambda})$ does not depend on $\lambda$. Strongly $\lambda$-uniform compatible systems are clearly $\lambda$-uniform. If $Z_0$ is not connected we say that a compatible system is \textit{$\lambda$-uniform} (resp. \textit{strongly $\lambda$-uniform}) if it is $\lambda$-uniform (resp. strongly $\lambda$-uniform) over every connected component.
\end{definition}

\begin{theorem}
\label{normal-uniform:t}
Let $Z_0$ be a normal variety over $\F_q$. Every untwisted $E$-compatible system $\ucalVz$ over $Z_0$ is strongly $\lambda$-uniform.
\end{theorem}

\begin{proof}
After extending the base field, we may assume that $Z_0$ is geometrically connected. As said before, since $Z_0$ is a normal variety, each $\lambda$-component of $\ucalVz$ is pure of weight $0$, \cite[Théorème 1.6]{Del}. Also, if $U_0$ is the smooth locus of $Z_0$, the \'etale fundamental group of $U$ maps surjectively onto the \'etale fundamental group of $Z$. Therefore, we have a canonical isomorphism $$H^0(Z,\calV_{\lambda})=H^0(U,\calV_{\lambda}|_U)$$ for every $\lambda$.  Applying \cite[Cor. VI.3]{Laf} to the $E$-compatible system $\{\calV_{\lambda,0}|_{U_0}\}_{\lambda\in |E|}$, we obtain the desired result.
\end{proof}

If we do not assume $Z_0$ normal, Theorem \ref{normal-uniform:t} becomes false in general (Example \ref{noda-curv:ex}). The issue is the lack of a Chebotarev density theorem for non-normal varieties. In what follows we want to understand if a weaker variant of Theorem \ref{normal-uniform:t} is still true for singular varieties.

\begin{definition}
\label{uniform-morphisms:d}
Let $f_0:Z_0\to Y_0$ be a morphism of varieties over $\Fq$. We say that $f_0$ is a \textit{$\lambda$-uniform morphism} if for every untwisted compatible system $\underline{\calV_0}$ over $Y_0$, the pullback $f_0^*\underline{\calV_0}$ is $\lambda$-uniform. If $f_0$ is a closed immersion we say that $Z_0$ is a \textit{$\lambda$-uniform subvariety} of $Y_0$.
\end{definition}
\begin{conjecture}\label{unif:c}
	Let $Y_0$ and $Z_0$ be varieties over $\Fq$. If $Y_0$ is normal, every morphism $f_0:Z_0\to Y_0$ is $\lambda$-uniform.
\end{conjecture}

%\begin{proposition}
%	\label{unif-on-fini-orde-dete:l}
%If $Y_0$ is smooth, a morphism $f_0:Z_0\to Y_0$ is $\lambda$-uniform if and only if for every untwisted compatible system $\underline{\calV_0}$ over $Y_0$ with every $\lambda$-component absolutely irreducible, the pullback $f_0^*\underline{\calV_0}$ is $\lambda$-uniform.
%\end{proposition}

%\begin{proof}Suppose that $\underline{\calV_0}$ is a pure $E$-compatible system over $Y_0$. Thanks to \cite[Corollary 3.5.2.(ii)]{Dad}, every component of $\underline{\calV_0}$ is geometrically semi-simple. Therefore, if we take the semi-simplification for each component we do not change the isomorphism class of the lisse sheaves over $Y$. We write $\underline{\calV_0}^{\mathrm{ss}}$ for this new compatible system and we choose a finite place $\lambda\nmid p$. After possibly extending $E$, the $\lambda$-component $\calV_{\lambda,0}^{\mathrm{ss}}$ is a direct sum of absolutely irreducible Weil lisse sheaf.  may For every  every pure lisse sheaf over $Y_0$ is geometrically semi-simple. Therefore, after taking semi-simplification, we can test the condition on those pure compatible systems over $Y_0$ which are semi-simple. Using \cite[Proposition 1.3.4]{Weil2}, we can also reduce to the case when $\underline{\calV_0}$ is irreducible with finite order determinant. In particular, it is enough to work with étale lisse sheaves, ???
	
%\end{proof}

If $Y_0$ is smooth we have an equivalent definition of a $\lambda$-uniform morphism. This other definition uses the \textit{Grothendieck semirings} of Weil lisse sheaves (cf. \cite[§4.3.3]{Dad}).
\begin{notation}
As in [\textit{ibid.}], we write $\KK(\calV_{\lambda,0})$ for the Grothendieck semiring of $\langle\calV_{\lambda,0} \rangle$, the Tannakian category spanned by $\calV_{\lambda,0}$. We also use an analogous notation for the lisse sheaf $\calV_\lambda$ over $Y$. Besides, we write $\KK^+(\calV_{\lambda})^0\subseteq \KK^+(\calV_{\lambda})$ for the sub-semiring of (geometrically) trivial lisse sheaves in $\langle\calV_{\lambda} \rangle$.
\end{notation}
If $\ucalVz$ is a split untwisted $E$-compatible system over a smooth variety, thanks to the companions conjecture, for every $\lambda,\lambda'\in |E|_{\neq p}$ we have a natural isomorphism of semirings $\KK^+(\calV_{\lambda,0})\iso\KK^+(\calV_{\lambda',0}),$ \cite[Proposition 4.3.6]{Dad}.
\begin{proposition}
	If $Y_0$ is a smooth variety, a morphism $f_0:Z_0\to Y_0$ is $\lambda$-uniform if and only if for every split untwisted $E$-compatible system $\underline{\calV_0}$ over $Y_0$ the preimage of $\KK^+(f^*\calV_{\lambda})^0$ via $\KK^+(\calV_{\lambda,0})\to \KK^+(f^*\calV_{\lambda})$ does not depend on $\lambda$.
\end{proposition}
\begin{proof}
	Note that by \cite[Corollary 3.5.2.(ii)]{Dad}, every $\lambda$-component of $\underline{\calV_0}$ is geometrically semi-simple. Therefore, if we take the semi-simplification of each $\lambda$-component we do not change the isomorphism classes of the lisse sheaves over $Y$. This shows that we may assume that each $\lambda$-component of $\underline{\calV_0}$ is semi-simple.
	Next, we observe that the preimage of $\KK^+(f^*\calV_{\lambda})^0$ is precisely the sub-semiring of $\KK^+(\calV_{\lambda,0})$ given by those semi-simple Weil lisse sheaves in $\langle\calV_{\lambda,0} \rangle$ which are geometrically trivial over $Z_0$.

	Suppose that the preimage of $\KK^+(f^*\calV_{\lambda})^0$ does not depend on $\lambda$. Looking at the class of $[\calV_{\lambda,0}]\in \KK^+(\calV_{\lambda,0})$ when $\lambda$ varies we deduce that $f_0$ is $\lambda$-uniform. For the converse, if $\calV_{\lambda,0}$ is untwisted, every object in $\langle\calV_{\lambda,0} \rangle$ is untwisted as well, \cite[Theorem 3.4.7]{Dad}. Therefore, every $\calW_{\lambda,0}\in \langle\calV_{\lambda,0} \rangle$ sits in an untwisted compatible system $\underline{\calW_0}$ over $Y_0$ which by assumption is $\lambda$-uniform when restricted to $Z_0$. This concludes the proof.
\end{proof}

\subsection{Homotopic invariance}

Let us look more closely at $\lambda$-uniform morphisms by analysing the relation with the induced morphism on fundamental groups. 

\subsubsection{}

Let $f_0:Z_0\to Y_0$ be a morphism of geometrically connected varieties over $\Fq$ with $Y_0$ normal. If we choose a geometric point $z$ of $Z_0$ we have a morphism $$\pi_1^{\et}(Z,z)\xrightarrow{f_{*}} \pi_1^{\et}(Y_0,f(z)).$$
	For every étale compatible system $\underline{\calV_0}$ over $Y_0$ we denote by $\{\rho_{\lambda,0}\}_{\lambda \in |E|_{\neq p}}$ the associated family of $\ell$-adic representations of $\pi_1^{\et}(Y_0,f(z))$. Let $\overline{\Image(f_*)}$ be the smallest normal closed subgroup of $\pi_1^{\et}(Y_0,f(z))$ containing the image of $f_*$. The following lemma is a direct consequence of the definition of a $\lambda$-uniform morphism.

\begin{lemma}\label{uniform-group-property:l}A morphism $f_0$ is $\lambda$-uniform if and only if for every untwisted compatible system over $Y_0$, if $\overline{\Image(f_*)}\subseteq \Ker(\rho_{\lambda,0})$ for one $\lambda$ then the same is true for every other $\lambda\in |E|$. In particular, the property of a morphism of being $\lambda$-uniform depends only on the inclusion $\overline{\Image(f_*)}\subseteq \pi_1^{\et}(Y_0,y)$ as topological groups together with the assignment of the conjugacy classes of the Frobenii at closed points of $\pi_1^{\et}(Y_0,y)$ and their degrees.
\end{lemma}

As a consequence of the previous lemma, we prove an “homotopic invariance” of $\lambda$-uniformity. Let $T_0$ and $S_0$ be geometrically connected varieties over $\F_q$ and $h_0: T_0\to S_0$ a proper and flat morphism with connected and reduced geometric fibres. Let $s_0$ and $s'_0$ be closed points of $S_0$ and write $\iota_0:Z_0\hookrightarrow T_0$ and $\iota'_0:Z'_0\hookrightarrow T_0$ for the closed immersions of the fibres of $h_0$ above $s_0$ and $s'_0$ respectively.
\begin{theorem}
\label{homotopy:t}For every morphism $\widetilde{f}_0: T_0\to Y_0$, the restriction $f_0:=\widetilde{f}_0|_{Z_0}$ is $\lambda$-uniform if and only if $f'_0:=\widetilde{f}_0|_{Z'_0}$ is $\lambda$-uniform.
\end{theorem}

\begin{proof}
Let $z$ and $z'$ be geometric points of $Z_0$ and $Z'_0$ respectively. By \cite[Tag 0C0J]{Stacks}, we have exact sequences $$\pie(Z,z)\xrightarrow{\iota_{*}} \pie(T,z)\xrightarrow{h_{*}} \pie(S,h(z))\to 1$$
$$\pie(Z',z')\xrightarrow{\iota'_{*}} \pie(T,z')\xrightarrow{h_{*}} \pie(S,h(z'))\to 1.$$
The choice of an étale path $\gamma$ joining $z$ with $z'$ induces isomorphisms $\gamma:\pie(T,z)\iso \pie(T',z')$ and $h_*(\gamma):\pie(S,h(z))\iso \pie(S',h(z'))$. Thanks to the two exact sequences this implies that $\gamma$ restricts to an isomorphism $\overline{\Image(\iota_*)}\iso \overline{\Image(\iota'_*)}$. In turn, this implies that the induced isomorphism $f_*(\gamma):\pi_1^{\et}(Y_0,f(z))\iso \pi_1^{\et}(Y_0,f'(z'))$ restricts to an isomorphism $\overline{\Image(f_*)}\iso \overline{\Image(f'_*)}$. By construction, $f_*(\gamma)$ respects the conjugacy classes of Frobenii at closed points and their degrees. We conclude applying Lemma \ref{uniform-group-property:l}.
\end{proof}

\section{Some examples}
In this section, we verify Conjecture \ref{unif:c} in some cases. Note that by virtue of Theorem \ref{normal-uniform:t} we already know the conjecture when $Z_0$ is normal. 
\subsection{Semi-stable curves}
\label{semi-stab-curv:ss}
Let $Z_0$ be a connected semi-stable curve over $\F_q$. Write $Z^{(i)}$ where ${1\leq i \leq n}$ for the irreducible components of $Z$. Suppose that for every $i$, the irreducible component $Z^{(i)}$ is smooth. Let $z$ be a geometric point of $Z$ and for every $1\leq i \leq n$, let $z^{(i)}$ be a generic geometric point of $Z^{(i)}$. We denote by $\Gamma$ the dual graph of $Z$ and by $P$ the point of $\Gamma$ associated to the connected component where $z$ lies. 
\begin{lemma}[\cite{Stix}]
	\label{Stix:l}
The choice of étale paths $\{\gamma^{(i)}\}_{1\leq i \leq n}$ joining $z$ to $z^{(i)}$ for every $i$ determines an isomorphism $$\pie(Z,z)\simeq \pie(Z^{(1)},z^{(1)})*\dots*\pie(Z^{(n)},z^{(n)})* \pie(\Gamma,P)^\wedge,$$ where ${\pi}_1(\Gamma,P)^\wedge$ is the profinite completion of the topological fundamental group of $\Gamma$. 
\end{lemma}

\begin{theorem}\label{tree:t}
If $\Gamma$ is a tree, every untwisted compatible system over $Z_0$ is $\lambda$-uniform.
\end{theorem}
\begin{proof}
Let $\ucalVz$ be an untwisted compatible system over $Z_0$. If $\calV_{\lambda,0}$ is geometrically trivial for one $\lambda$, then it remains geometrically trivial when restricted to every irreducible component of $Z_0$. By Theorem \ref{normal-uniform:t}, for every other $\lambda'\in |E|_{\neq p}$, the restriction of $\calV_{\lambda',0}$ to every irreducible component of $Z_0$ is geometrically trivial as well. By Lemma \ref{Stix:l}, since $\Gamma$ is a tree, the geometric étale fundamental group of $Z_0$ is generated by the étale fundamental groups of the irreducible components $Z^{(i)}$. This shows that for every $\lambda$ the Weil lisse sheaf $\calV_{\lambda',0}$ is geometrically trivial over $Z_0$ and this yields the desired result.
\end{proof}

\begin{example}
	\label{noda-curv:ex}
If $Z_0$ is an irreducible split nodal cubic curve over $\Fq$ with nodal point $z_0$, then $\pie(Z,z)$ is isomorphic to $\hZ$. In addition, the action of $\Gal(\F/\Fq)$ on $\pie(Z,z)$ is trivial. Therefore, $\pie(Z_0,z)$ is isomorphic to $\hZ\times \Gal(\F/\Fq)$, where the embedding $\Gal(\F/\Fq)\subseteq \pie(Z_0,z)$ is induced by the closed immersion of $z_0\hookrightarrow Z_0$. The Frobenius elements of $\pie(Z_0,z)$ correspond via this isomorphism to elements of $\hZ\times \Gal(\F/\Fq)$ of the form $(0,F^d)$, where $F$ is the geometric Frobenius of $\Fq$ and $d$ is some positive integer. This implies that every pair of étale lisse sheaves over $Z_0$ which are trivial at $z_0$ are $\Q$-compatible with all the eigenvalues at closed points equal to $1$. 

On the other hand, for every prime number $\ell\neq p$ and every continuous automorphism $\alpha$ of $\overline{\Z}_\ell^{\oplus r}$, where $\overline{\Z}_\ell$ is the ring of integers of $\Qlbar$, there exists an étale lisse $\Qlbar$-sheaf of this type such that the induced $\ell$-adic representation sends $(1,\mathrm{id})$ to $\alpha$. In particular, we may take $r=1$ and as $\alpha$ the multiplication by a root of unit. This construction produces lots of examples of untwisted compatible systems over $Z_0$ which are not $\lambda$-uniform.
\end{example}
\subsection{Finite monodromy}
%\begin{rema}
%If $Z_0$ is not normal then previous theorem is not true. For example if $Z_0$ is the nodal curve over $\F_q$ and $z$ is a geometric point of $Z_0$ then  the étale fundamental group of $\hZ \times \hZ$ 
%\end{rema}
As in Theorem \ref{normal-uniform:t}, we may use the theory of weights in order to prove that certain morphisms are $\lambda$-uniform. This strategy needs strong finiteness conditions.
\begin{theorem}
\label{finite:t}
Let $f_0: Z_0 \to Y_0$ be a morphism of geometrically connected varieties over $\Fq$ with $Y_0$ normal. Let $z$ be a geometric point of $Z_0$. The morphism $f_0$ is $\lambda$-uniform in the following cases.
\begin{itemize}
\item [{\rm (i)}] If the smallest closed normal subgroup of $\pi_1^{\et}(Y_0,f(z))$ containing the image of $\pi_1^{\et}(Z,z)$ is an open subgroup of $\pi_1^{\et}(Y,f(z))$.
\item [{\rm (ii)}] If $\pi_1^{\et}(Y,f(z))$ contains an open solvable profinite subgroup.
\end{itemize}

\end{theorem}
\begin{proof}Let $\underline{\calV_0}$ be an untwisted compatible system over $Y_0$. For every $\lambda\in |E|_{\neq p}$, write $G_{\lambda}$ for the geometric monodromy group of $\calV_{\lambda,0}$.\\
	
\begin{itemize}
\item [(i)]
If $f_0^*(\calV_{\lambda,0})$ is geometrically trivial over $Z_0$ for one $\lambda$, then $\rho_{\lambda}$ is trivial when restricted to the image of $\pi_1^{\et}(Z,z)$ in $\pi_1^{\et}(Y,f(z))$. Thanks to the hypothesis, we deduce that $\rho_{\lambda}$ factors through a finite quotient of $\pi_1^{\et}(Y,f(z))$. This implies that $G_{\lambda}$ is a finite algebraic group. By \cite[Proposition 2.2]{LP5}, the subgroups $$\Ker(\rho_{\lambda'})\cap \pi_1^{\et}(Y,f(z))\subseteq \pi_1^{\et}(Y,f(z))$$ are all equal when $\lambda'$ varies in $|E|_{\neq p}$. Since the image of $\pi_1^{\et}(Z,z)$ in $\pi_1^{\et}(Y,f(z))$ is contained in $\Ker(\rho_{\lambda})$, it is also contained in $\Ker(\rho_{\lambda'})$ for every $\lambda'\in |E|_{\neq p}$. Therefore, the lisse sheaf $f_0^*(\calV_{\lambda',0})$ is geometrically trivial for every $\lambda'\nmid p$, as we wanted.\\

\item [(ii)]
By \cite[Corollary 3.4.10]{Dad}, the algebraic groups $G_\lambda$ are all semi-simple because $\underline{\calV_0}$ is pure. On the other hand, thanks to the assumption that $\pi_1^{\et}(Y,f(z))$ is solvable, we know that all the algebraic groups $G_\lambda$ are solvable as well. Combining these two properties we deduce that each algebraic group $G_\lambda$ is finite and we can proceed as in the previous case.
\end{itemize}

\end{proof}
\begin{corollary}
A dominant morphism $f_0:Z_0\to Y_0$ is $\lambda$-uniform. In particular, if $Y_0$ is a smooth curve, every morphism to $Y_0$ is $\lambda$-uniform.
\end{corollary}
\begin{proof} Let $\eta\hookrightarrow Y$ be the generic point of $Y$ and write $Z_{\eta}$ the preimage of $\eta$ via $f$. Since $f$ is dominant, the scheme $Z_{\eta}$ is a non-empty variety over the function field of $Y$. Fix a closed point $\eta'$ of $Z_{\eta}$ and choose a geometric point $\overline{\eta}$ over $\eta'$. We have the following commutative diagram
	
			\begin{center}
		\begin{tikzcd}
		\pie(\eta',\overline{\eta}) \arrow[hook,r] \arrow[d] &  \pie(\eta,f(\overline{\eta}))  \arrow[d, two heads]\\
		\pi_1^{\et}(Z,\overline{\eta}) \arrow[r] & \pi_1^{\et}(Y,f(\overline{\eta})).
		\end{tikzcd}
	\end{center}
Note that $\pie(\eta',\overline{\eta})$ maps to a finite index subgroup of $\pie(\eta,f(\overline{\eta}))$. Therefore, the image of $\pi_1^{\et}(Z,\overline{\eta})$ in $\pi_1^{\et}(Y,f(\overline{\eta}))$ has finite index as well. Thanks to this, we may apply Theorem \ref{finite:t}.(i) to conclude. 
\end{proof}

\subsection{Final comments}\label{fina-comm:ss}
Apart from the cases presented above, it seems difficult to prove in general that a closed subvariety of a smooth variety is $\lambda$-uniform. Note that the property of a Weil lisse sheaf of being geometrically trivial can be thought as the combination of two properties: the property of being geometrically unipotent and the one of being geometrically finite\footnote{For us, a geometrically unipotent (resp. finite) Weil lisse sheaf is a Weil lisse sheaf with unipotent (resp. finite) geometric monodromy group.}. Let us focus on the second one. 
\begin{definition}
We say that a subvariety $Z_0\subseteq Y_0$ is \textit{pseudo-$\lambda$-uniform} if it satisfies the analogous condition as a $\lambda$-uniform subvariety where “geometrically \textit{trivial} Weil lisse sheaves” are replaced by “geometrically \textit{finite} Weil lisse sheaves”.
\end{definition}
 In the following example, we present a concrete class of subvarieties where the pseudo-$\lambda$-uniformity is not immediate and, at the best of our knowledge, it is not known.

\begin{example}[de Jong]
	\label{dJ:ex}
	For an integer $n\geq 3$ such that $(n,p)=1$, we write $Y^{(n)}$ for the moduli scheme of principally polarized abelian surfaces over $\F$ with a symplectic level-$n$-structure. Let $Z^{(n)}\subseteq Y^{(n)}$ be the supersingular locus of $Y^{(n)}$. Since we are working with abelian surfaces, this coincides with the Moret--Bailly locus of $Y^{(n)}$, which is a union of copies of $\PP^1_{\F}$. Therefore, by \cite[Proposition 7.3]{Oor01}, the variety $Z^{(n)}$ is connected for every choice of $n$. Let $N$ be a positive multiple of $n$ which is also prime to $p$. The preimage of the natural finite étale Galois cover $Y^{(N)}\to Y^{(n)}$ is $Z^{(N)}$. As $Z^{(N)}$ is connected, the restriction $Z^{(N)}\to Z^{(n)}$ is a finite étale Galois cover with the same Galois group as $\Gal(Y^{(N)}/Y^{(n)})$. If $z$ is a geometric point of $Z^{(n)}$, we have the following commutative diagram.
	\begin{center}
		\begin{tikzcd}
		\pi_1^{\et}(Z^{(n)},z) \arrow{r} \arrow[d,two heads] &  \pi_1^{\et}(Y^{(n)},z)\arrow[d,two heads]\\
		\Gal(Z^{(N)}/Z^{(n)})   & \arrow[l,"="'] \Gal(Y^{(N)}/Y^{(n)}).\
		\end{tikzcd}
	\end{center}
	When $N$ goes to infinity, the cardinality of $\Gal(Y^{(N)}/Y^{(n)})$ goes to infinity as well. This implies that the image of $\pi_1^{\et}(Z^{(n)},z)\to \pi_1^{\et}(Y^{(n)},z)$ is infinite. The varieties $Z^{(n)}$ and $Y^{(n)}$ descend to geometrically connected varieties $Z_0^{(n)}$ and $Y_0^{(n)}$ over some finite field $\F_q$. The restriction of a Weil lisse sheaf over $Y_0^{(n)}$ is not in general geometrically finite over $Z_0^{(n)}$. We cannot prove in this case that $Z_0^{(n)}\subseteq Y_0^{(n)}$ is pseudo-$\lambda$-uniform. Nonetheless, we think this might be a nice concrete example to analyse.
	
	\end{example}

\bibliographystyle{ams-alpha}

\end{document}